\documentclass{article}
\usepackage{amsfonts}
\usepackage{amsmath}

\newtheorem{theorem}{Theorem}
\newtheorem{acknowledgement}[theorem]{Acknowledgement}
\newtheorem{algorithm}[theorem]{Lemma and Definition}

\newtheorem{definition}[theorem]{Definition}
\newtheorem{example}[theorem]{Example}

\newtheorem{lemma}[theorem]{Lemma}
\newtheorem{notation}[theorem]{Notation}

\newtheorem{remark}[theorem]{Remark}

\newenvironment{proof}[1][Proof]{\noindent\textbf{#1.} }{\ \rule{0.5em}{0.5em}}
\input{tcilatex}

\begin{document}

\title{Injectivity and flatness of semitopological modules}
\author{Henri Bourl\`{e}s\thanks{%
SATIE, ENS Cachan/CNAM, 61 Avenue President Wilson, F-94230 Cachan, France
(Henri.Bourles@satie.ens-cachan.fr).}}
\maketitle

\begin{abstract}
The spaces $\mathcal{D}$, $\mathcal{S}$ and $\mathcal{E}^{\prime }$ over $%
\mathbb{R}
^{n}$ are known to be flat modules over $\mathbf{A}=%
\mathbb{C}
\left[ \partial _{1},...,\partial _{n}\right] $, whereas their duals $%
\mathcal{D}^{\prime }$, $\mathcal{S}^{\prime }$ and $\mathcal{E}$ are known
to be injective modules over the same ring. \ \ Let $\mathbf{A}$ be a
Noetherian $\mathbf{k}$-algebra ($\mathbf{k}=%
\mathbb{R}
$ or $%
\mathbb{C}
$). \ The above observation leads us to study in this paper the link
existing between the flatness of an $\mathbf{A}$-module $E$ which is a
locally convex topological $\mathbf{k}$-vector space and the injectivity of
its dual. \ We show that, for dual pairs $\left( E,E^{\prime }\right) $
which are $\left( \mathcal{K}\right) $ over $\mathbf{A}$--a notion which is
explained in the paper--, injectivity of $E^{\prime }$ is a stronger
condition than flatness of $E$. \ A preprint of this paper (dated September
2009) has been quoted and discussed by Shankar \cite{Shankar}.
\end{abstract}

\sloppy

\section{Introduction}

Consider the spaces $\mathcal{D}$, $\mathcal{S}$ and $\mathcal{E}^{\prime }$
over $%
\mathbb{R}
^{n}$, as well as their duals $\mathcal{D}^{\prime }$, $\mathcal{S}^{\prime
} $ and $\mathcal{E}$. \ Ehrenpreis \cite{Ehrenpreis-book}, Malgrange \cite%
{Malgrange-fund-princ}, \cite{Malgrange-Bourbaki} and Palamodov \cite%
{Palamodov} proved that $\mathcal{D}$, $\mathcal{S}$ and $\mathcal{E}%
^{\prime }$ are flat modules over $\mathbf{A}=%
\mathbb{C}
\left[ \partial _{1},...,\partial _{n}\right] $ whereas $\mathcal{D}^{\prime
}$, $\mathcal{S}^{\prime }$ and $\mathcal{E}$ are injective over $\mathbf{A}$%
. \ If $F$ is any of these modules, all maps $F\rightarrow F:x\mapsto a\,x$ $%
\left( a\in \mathbf{A}\right) $\ are continuous; using Pirkovskii' s
terminology (\cite{Pirkovskii}, p. 5), this means that $F$ is \emph{%
semitopological}. \ This observation leads to wonder whether there exists a
link between the injectivity of a semitopological $\mathbf{A}$-module and
the flatness of its dual. \ The existence of such a link is studied in this
paper.

\section{Preliminaries}

\begin{notation}
In what follows, $\mathbf{A}$ is a Noetherian domain (not necessarily
commutative) which is a $\mathbf{k}$-algebra ($\mathbf{k}=%
\mathbb{R}
$ or $%
\mathbb{C}
$).
\end{notation}

Let $E$, $E^{\prime \text{ }}$be two $\mathbf{k}$-vector spaces. \ Assume
that $E^{\prime }$ is a left $\mathbf{A}$-module and that there exists a
nondegenerate bilinear form $\left\langle -,-\right\rangle :E\times
E^{\prime }\rightarrow \mathbf{k}$. \ Then $E$ and $E^{\prime }$ are locally
convex topological vector spaces endowed with the weak topologies $\sigma
\left( E,E^{\prime }\right) $ and $\sigma \left( E^{\prime },E\right) $
defined by $\left\langle -,-\right\rangle $; the pair $\left( E,E^{\prime
}\right) $ is called \emph{dual }(with respect to the bilinear form $%
\left\langle -,-\right\rangle $).

Assume that the left $\mathbf{A}$-module $E^{\prime }$ (written $_{\mathbf{A}%
}E^{\prime }$) is semitopological for the topology $\sigma \left( E^{\prime
},E\right) $. \ Then the $\mathbf{k}$-vector space $E$ becomes a right $%
\mathbf{A}$-module (written $E_{\mathbf{A}}$), setting 
\begin{equation}
\left\langle x\,a,x^{\prime }\right\rangle =\left\langle x,a\,x^{\prime
}\right\rangle  \label{duality-bracket}
\end{equation}%
for any $x\in E$, $x^{\prime }\in E^{\prime }$ and $a\in \mathbf{A}$, and it
is obviously semitopological, i.e., all maps $E\rightarrow E:x\mapsto x\,a$ $%
\left( a\in \mathbf{A}\right) $\ are continuous. \ Conversely, one can
likewise prove that if the right $\mathbf{A}$-module $E_{\mathbf{A}}$ is
semitopological for the topology $\sigma \left( E,E^{\prime }\right) $, then 
$_{\mathbf{A}}E^{\prime }$ is semitopological for the topology $\sigma
\left( E^{\prime },E\right) $. \ By $\left( \ref{duality-bracket}\right) $,
the transpose of the left multiplication by $a\in \mathbf{A}$, denoted by $%
a\bullet :E^{\prime }\rightarrow E^{\prime }$, is the right multiplication
by $a$, denoted by $\bullet a:E\rightarrow E$.

\begin{notation}
In what follows, $\left( E,E^{\prime }\right) $ is a dual pair and $E_{%
\mathbf{A}}$ (or equivalently $_{\mathbf{A}}E^{\prime }$) is a
semitopological module.
\end{notation}

The duality bracket $\left\langle -,-\right\rangle $ is extended to an
obvious way to $E^{1\times k}\times \left( E^{\prime }\right) ^{k}$; then $%
\left( E^{1\times k},\left( E^{\prime }\right) ^{k}\right) $ is again a dual
pair. \ Let $P\in \mathbf{A}^{q\times k}$; this matrix determines a
continuous linear map $P\bullet :\left( E^{\prime }\right) ^{k}\rightarrow
\left( E^{\prime }\right) ^{q}:x^{\prime }\mapsto P\,x^{\prime }$, the
transpose of which is $\bullet P:E^{1\times q}\times E^{1\times k}:x\mapsto
x\,P$.

\begin{example}
Let $E^{\prime }$ be the space of distributions $\mathcal{D}^{\prime }$, $%
\mathcal{S}^{\prime }$ or $\mathcal{E}^{\prime }$ over $%
\mathbb{R}
^{n}$ and $E$ the associated space of test functions. \ From the above, the
transpose of $\partial _{i}\bullet :E^{\prime }\rightarrow E^{\prime }$ is $%
\bullet \partial _{i}:\mathcal{E}\rightarrow \mathcal{E}$, and for any $T\in
E^{\prime }$, $\varphi \in E$, $\left\langle \varphi \,\partial
_{i},T\right\rangle =\left\langle \varphi ,\partial _{i}\,T\right\rangle $.
\ Since $\left\langle \varphi ,\partial _{i}\,T\right\rangle =-\left\langle
\partial _{i}\,\varphi ,T\right\rangle $, one has $\varphi \,\partial
_{i}=-\partial _{i}\,\varphi $ $\left( \varphi \in E\right) $, i.e., $%
\bullet \partial _{i}=-\partial _{i}\bullet $.
\end{example}

Consider the following sequences where $P_{1}\in \mathbf{A}^{k_{1}\times
k_{2}},P_{2}\in \mathbf{A}^{k_{2}\times k_{3}}$:%
\begin{equation}
\mathbf{A}^{1\times k_{1}}\overset{\bullet P_{1}}{\longrightarrow }\mathbf{A}%
^{1\times k_{2}}\overset{\bullet P_{2}}{\longrightarrow }\mathbf{A}^{1\times
k_{3}},  \label{exact-seq-1}
\end{equation}%
\begin{equation}
E^{1\times k_{1}}\overset{\bullet P_{1}}{\longrightarrow }E^{1\times k_{2}}%
\overset{\bullet P_{2}}{\longrightarrow }E^{1\times k_{3}},
\label{exact-sequence-2}
\end{equation}%
\begin{equation}
\left( E^{\prime }\right) ^{k_{3}}\overset{P_{2}\bullet }{\longrightarrow }%
\left( E^{\prime }\right) ^{k_{2}}\overset{P_{1}\bullet }{\longrightarrow }%
\left( E^{\prime }\right) ^{k_{1}}.  \label{exact-sequence-3}
\end{equation}

The facts recalled below are classical:

\begin{lemma}
\label{Lemma-recalled}(i) The module $E_{\mathbf{A}}$ is flat if, and only
if whenever $\left( \ref{exact-seq-1}\right) $ is exact, $\left( \ref%
{exact-sequence-2}\right) $, deduced from $\left( \ref{exact-seq-1}\right) $
using the functor $E\tbigotimes\nolimits_{\mathbf{A}}-$, is again exact (%
\cite{Palamodov}, Part I, \S I.3, Prop. 5).\newline
(ii) The module $_{\mathbf{A}}E^{\prime }$ is injective if, and only if
whenever $\left( \ref{exact-seq-1}\right) $ is exact, $\left( \ref%
{exact-sequence-3}\right) $, deduced from $\left( \ref{exact-seq-1}\right) $
using the functor $\func{Hom}_{\mathbf{A}}\left( -,E^{\prime }\right) $, is
again exact (\cite{Palamodov}, Part I, \S I.3, Prop. 9).\newline
(iii) For any matrix $P_{2}\in \mathbf{A}^{k_{2}\times k_{3}}$, there exist
a natural integer $k_{1}$ and a matrix $P_{1}\in \mathbf{A}^{k_{1}\times
k_{2}}$ such that $\left( \ref{exact-seq-1}\right) $ is exact. \ Conversely,
given a matrix $P_{1}\in \mathbf{A}^{k_{1}\times k_{2}}$, there exists a
matrix $P_{2}\in \mathbf{A}^{k_{2}\times k_{3}}$ such that $\left( \ref%
{exact-seq-1}\right) $ is exact\ if, and only if $\func{coker}_{\mathbf{A}%
}\left( \bullet P_{1}\right) =\mathbf{A}^{1\times k_{2}}/\left( \mathbf{A}%
^{1\times k_{1}}\,P_{2}\right) $ is torsion-free (see, e.g., \cite%
{Bourles-Oberst-1}, Lemma 2.15).\newline
(iv) The following equalities hold (\cite{Bourbaki-EVT}, \S IV.6, Corol. 2
of Prop. 6):%
\begin{eqnarray*}
\ker _{E^{\prime }}\left( P_{1}\bullet \right) &=&\left( \func{im}_{E}\left(
\bullet P_{1}\right) \right) ^{0}, \\
\overline{\func{im}_{E^{\prime }}\left( P_{2}\bullet \right) } &=&\left(
\ker _{E}\left( \bullet P_{2}\right) \right) ^{0}
\end{eqnarray*}%
where $\left( .\right) ^{0}$ is the polar of $\left( .\right) $.
\end{lemma}

Consider the sequence involving $2+n$ maps $\bullet P_{i}$ $\left( 1\leq
i\leq 2+n\right) $%
\begin{equation}
\mathbf{A}^{1\times k_{1}}\overset{\bullet P_{1}}{\longrightarrow }\mathbf{A}%
^{1\times k_{2}}\overset{\bullet P_{2}}{\longrightarrow }\mathbf{A}^{1\times
k_{3}}\overset{}{\longrightarrow }...\overset{\bullet P_{2+n}}{%
\longrightarrow }\mathbf{A}^{1\times k_{3+n}}  \label{extended-exact-seq}
\end{equation}%
where $n\geq 0$.

\begin{definition}
\label{def-injective-torsionfree}The module $_{\mathbf{A}}E^{\prime }$ is
called $n$-injective if whenever $\left( \ref{extended-exact-seq}\right) $
is exact, $\left( \ref{exact-sequence-3}\right) $ is again exact.
\end{definition}

The following is obvious:

\begin{lemma}
(i) If the module $_{\mathbf{A}}E^{\prime }$ is $n$-injective $\left( n\geq
0\right) $, then it is $n^{\prime }$-injective for all integers $n^{\prime }$
such that $n^{\prime }\geq n$.\newline
(ii) The module $_{\mathbf{A}}E^{\prime }$ is $0$-injective if, and only if
it is injective.
\end{lemma}

\begin{lemma}
\label{lemma-cl(images)-kernel}(1) If $\left( \ref{exact-sequence-2}\right) $
is exact, then $\overline{\func{im}_{E^{\prime }}\left( P_{2}\bullet \right) 
}=\ker _{E^{\prime }}\left( P_{1}\bullet \right) $.\newline
(2) If $\left( \ref{exact-sequence-3}\right) $ is exact, then $\overline{%
\func{im}_{E}\left( \bullet P_{1}\right) }=\ker _{E}\left( \bullet
P_{2}\right) $.
\end{lemma}

\begin{proof}
(1) If $\left( \ref{exact-sequence-2}\right) $ is exact, then $\ker
_{E}\left( \bullet P_{2}\right) =\func{im}_{E}\left( \bullet P_{1}\right) $,
therefore $\left( \ker _{E}\left( \bullet P_{2}\right) \right) ^{0}=\left( 
\func{im}_{E}\left( \bullet P_{1}\right) \right) ^{0}$ with $\left( \ker
_{E}\left( \bullet P_{2}\right) \right) ^{0}=\overline{\func{im}_{E^{\prime
}}\left( P_{2}\bullet \right) }$ and $\left( \func{im}_{E}\left( \bullet
P_{1}\right) \right) ^{0}=\ker _{E^{\prime }}\left( P_{1}\bullet \right) $.

(2) If $\left( \ref{exact-sequence-3}\right) $ is exact, then $\ker
_{E^{\prime }}\left( P_{1}\bullet \right) =\func{im}_{E^{\prime }}\left(
P_{2}\bullet \right) $, therefore $\left( \func{im}_{E}\bullet P_{1}\right)
^{0}=\func{im}_{E^{\prime }}\left( P_{2}\bullet \right) $, thus $\left( 
\func{im}_{E}\left( \bullet P_{1}\right) \right) ^{00}=\left( \func{im}%
_{E^{\prime }}\left( P_{2}\bullet \right) \right) ^{0}=\left( \overline{%
\func{im}_{E^{\prime }}\left( P_{2}\bullet \right) }\right) ^{0}=\left( \ker
_{E}\left( \bullet P_{2}\right) \right) ^{00}$, and $\overline{\func{im}%
_{E}\left( \bullet P_{1}\right) }=\ker _{E}\left( \bullet P_{2}\right) $ by
the bipolar theorem since $\ker _{E}\left( \bullet P_{2}\right) $ is closed.
\end{proof}

\section{Injectivity vs. flatness}

\begin{algorithm}
(1) Let $P\in \mathbf{A}^{k\times r}$; Conditions (i)-(iv) below are
equivalent:\newline
(i) $P\bullet :\left( E^{\prime }\right) ^{r}\rightarrow \left( E^{\prime
}\right) ^{k}$ is a strict morphism and so is also $\bullet P:E^{1\times
k}\rightarrow E^{1\times r}$;\newline
(ii) $P\bullet :\left( E^{\prime }\right) ^{r}\rightarrow \left( E^{\prime
}\right) ^{k}$ is a strict morphism with closed image (in $\left( E^{\prime
}\right) ^{k}$);\newline
(iii) $\bullet P:E^{1\times k}\rightarrow E^{1\times r}$ is a strict
morphism with closed image (in $E^{1\times r}$);\newline
(iv) both maps $\bullet P:E^{1\times k}\rightarrow E^{1\times r}$ and $%
P\bullet :\left( E^{\prime }\right) ^{r}\rightarrow \left( E^{\prime
}\right) ^{k}$ have a closed image.\newline
(2) The dual pair $\left( E,E^{\prime }\right) $ is said to be \emph{K\"{o}%
the} (or $\left( \mathcal{K}\right) $, for short)\ over $\mathbf{A}$ if for
any positive integers $k,r$ and any matrix $P\in \mathbf{A}^{k\times r}$,
the following condition holds: $\bullet P:E^{1\times k}\rightarrow
E^{1\times r}$ has a closed image if, and only if $P\bullet :\left(
E^{\prime }\right) ^{r}\rightarrow \left( E^{\prime }\right) ^{k}$ has a
closed image.
\end{algorithm}

\begin{proof}
(1): see, e.g., (\cite{Kothe-TVS-II}, \S 32.3).
\end{proof}

\begin{remark}
\label{example-QF}(1) The dual pair $\left( E,E^{\prime }\right) $ is not
necessarily $\left( \mathcal{K}\right) $ over $\mathbf{A}$ by (\cite%
{Bourbaki-EVT}, \S II.6, Remark 2 after Corol. 4 of Prop. 7); see, also, (%
\cite{Dierolf}, Prop. 2.3).\newline
(2) Assume that $E$ is a Fr\'{e}chet space (\textit{e.g.}, $E=\mathcal{S}$), 
$E^{\prime }$ is its dual and $\left\langle -,-\right\rangle $ is the
canonical duality bracket. \ Then for any integer $k$, $E^{1\times k}$ is
again a Fr\'{e}chet space, and the dual pair $\left( E,E^{\prime }\right) $
is $\left( \mathcal{K}\right) $ over $\mathbf{A}$ by (\cite{Bourbaki-EVT}, 
\S IV.4, Theorem 1).\newline
(3) Likewise, if $E$ is the dual of a \emph{reflexive} Fr\'{e}chet space,
then the dual pair $\left( E,E^{\prime }\right) $ is $\left( \mathcal{K}%
\right) $ over $\mathbf{A}$. \ Indeed, let $E=F^{\prime }$ where $F$ is a
reflexive Fr\'{e}chet space. \ If $\bullet P:\left( F^{\prime }\right)
^{1\times k}\rightarrow \left( F^{\prime }\right) ^{1\times r}$ has a closed
image, then by the above-quoted theorem $P\bullet :F^{r}\rightarrow F^{k}$
has a closed image and $F=F^{\prime \prime }=E^{\prime }.$ \ Conversely, if $%
P\bullet :\left( E^{\prime }\right) ^{r}\rightarrow \left( E^{\prime
}\right) ^{k}$ has a closed image, then $\bullet P:E^{1\times k}\rightarrow
E^{1\times r}$ has a closed image, for $E^{\prime }=F$ and $E=F^{\prime }.$%
\newline
(4) Whether the above holds when $E$ is an arbitrary $\left( \mathcal{LF}%
\right) $ space was mentioned in (\cite{Dieudonne-Schwartz}, \S 15.10) as
being an open question; to our knowledge, this question is still open today.
\end{remark}

\begin{lemma}
\label{lemma-E-prime-injective-strict}Let $P_{1}\in \mathbf{A}^{k_{1}\times
k_{2}}$.\newline
(i) Assume that $_{\mathbf{A}}E^{\prime }$ is injective. \ Then $\limfunc{im}%
_{E^{\prime }}\left( P_{1}\bullet \right) $ is closed (or equivalently, $%
\bullet P_{1}:E^{1\times k_{1}}\rightarrow E^{1\times k_{2}}$ is strict).%
\newline
(ii) Assume that $\func{coker}_{\mathbf{A}}\left( \bullet P_{1}\right) $ is
torsion-free and $E_{\mathbf{A}}$ is flat. \ Then $\limfunc{im}_{E}\left(
\bullet P_{1}\right) $ is closed (or equivalently, $P_{1}\bullet :\left(
E^{\prime }\right) ^{k_{2}}\rightarrow \left( E^{\prime }\right) ^{k_{1}}$
is strict).
\end{lemma}

\begin{proof}
(i): By Lemma \ref{Lemma-recalled}(iii), there exists a matrix $P_{0}\in 
\mathbf{A}^{k_{0}\times k_{1}}$ such that the sequence%
\begin{equation*}
\mathbf{A}^{1\times k_{0}}\overset{\bullet P_{0}}{\longrightarrow }\mathbf{A}%
^{1\times k_{1}}\overset{\bullet P_{1}}{\longrightarrow }\mathbf{A}^{1\times
k_{2}}
\end{equation*}%
is exact, and since $_{\mathbf{A}}E^{\prime }$ is injective, the sequence%
\begin{equation*}
\left( E^{\prime }\right) ^{k_{2}}\overset{P_{1}\bullet }{\longrightarrow }%
\left( E^{\prime }\right) ^{k_{1}}\overset{P_{0}\bullet }{\longrightarrow }%
\left( E^{\prime }\right) ^{k_{0}}
\end{equation*}%
is exact. \ Therefore, $\func{im}_{E^{\prime }}\left( P_{1}\bullet \right)
=\ker _{E^{\prime }}\left( P_{0}\bullet \right) $, thus $\func{im}%
_{E^{\prime }}\left( P_{1}\bullet \right) $ is closed, and $\bullet
P_{1}:E^{1\times k_{1}}\rightarrow E^{1\times k_{2}}$ is strict by (\cite%
{Kothe-TVS-II}, \S 32.3).

(ii): Since $\func{coker}_{\mathbf{A}}\left( \bullet P_{1}\right) $ is
torsion-free, by Lemma \ref{Lemma-recalled}(iii) there exists $P_{2}\in 
\mathbf{A}^{k_{2}\times k_{3}}$ such that the sequence $\left( \ref%
{exact-seq-1}\right) $ is exact. \ Since $E_{\mathbf{A}}$ is flat, the
sequence $\left( \ref{exact-sequence-2}\right) $ is exact. \ Therefore, $%
\func{im}_{E}\left( \bullet P_{1}\right) =\ker _{E}\left( \bullet
P_{2}\right) $ is closed, and $P_{1}\bullet :\left( E^{\prime }\right)
^{k_{2}}\rightarrow \left( E^{\prime }\right) ^{k_{1}}$ is strict by (\cite%
{Kothe-TVS-II}, \S 32.3).
\end{proof}

\begin{theorem}
\label{th-paper}Assume that the dual pair $\left( E,E^{\prime }\right) $ is $%
\left( \mathcal{K}\right) $ over $\mathbf{A}$.\newline
(1) If $_{\mathbf{A}}E^{\prime }$ is injective, then $E_{\mathbf{A}}$ is
flat.\newline
(2) Conversely, if $E_{\mathbf{A}}$ is flat, then $_{\mathbf{A}}E^{\prime }$
is $1$-injective.
\end{theorem}

\begin{proof}
(1) Assume that $_{\mathbf{A}}E^{\prime }$ is injective and $\left( \ref%
{exact-seq-1}\right) $ is exact. \ Then $\left( \ref{exact-sequence-3}%
\right) $\ is exact, which implies that $\overline{\func{im}_{E}\left(
\bullet P_{1}\right) }=\ker _{E}\left( \bullet P_{2}\right) $ according to
Lemma \ref{lemma-cl(images)-kernel}(2). \ By Lemma \ref%
{lemma-E-prime-injective-strict}(i), $\func{im}_{E^{\prime }}\left(
P_{1}\bullet \right) $ is closed. \ Since $\left( E,E^{\prime }\right) $ is $%
\left( \mathcal{K}\right) $, $\func{im}_{E}\left( \bullet P_{1}\right) $ is
also closed. \ Hence $\func{im}_{E}\left( \bullet P_{1}\right) =\ker
_{E}\left( \bullet P_{2}\right) ,$ i.e., $\left( \ref{exact-sequence-2}%
\right) $ is exact. \ This proves that $E_{\mathbf{A}}$ is flat.

(2) Assume $E_{\mathbf{A}}$ is flat and the sequence $\left( \ref%
{extended-exact-seq}\right) $ is exact with $n=1$. \ Then, the sequence%
\begin{equation*}
E^{1\times k_{1}}\overset{\bullet P_{1}}{\longrightarrow }E^{1\times k_{2}}%
\overset{\bullet P_{2}}{\longrightarrow }E^{1\times k_{3}}\overset{\bullet
P_{3}}{\longrightarrow }E^{1\times k_{4}}
\end{equation*}%
is exact. \ By Lemma \ref{lemma-cl(images)-kernel}(1) we obtain%
\begin{equation*}
\overline{\func{im}_{E^{\prime }}\left( P_{2}\bullet \right) }=\ker
_{E^{\prime }}\left( P_{1}\bullet \right) .
\end{equation*}%
In addition, $\func{im}_{E}\left( \bullet P_{2}\right) =\ker _{E}\left(
\bullet P_{3}\right) $, thus $\func{im}_{E}\left( \bullet P_{2}\right) $ is
closed, and since $\left( E,E^{\prime }\right) $ is $\left( \mathcal{K}%
\right) $, $\func{im}_{E^{\prime }}\left( P_{2}\bullet \right) $ is closed.
\ This proves that $\func{im}_{E^{\prime }}\left( P_{2}\bullet \right) =\ker
_{E^{\prime }}\left( P_{1}\bullet \right) $, i.e., the sequence $\left( \ref%
{exact-sequence-3}\right) $ is exact, and $_{\mathbf{A}}E^{\prime }$ is $1$%
-injective.
\end{proof}

\section{Concluding remarks}

Consider a dual pair $\left( E,E^{\prime }\right) $ which is $\left( 
\mathcal{K}\right) $ over the $\mathbf{k}$-algebra $\mathbf{A}$. \ As shown
by Theorem \ref{th-paper}, injectivity of $_{\mathbf{A}}E^{\prime }$ implies
flatness of $E_{\mathbf{A}}$. \ The converse does not hold, since flatness
of $E_{\mathbf{A}}$ only implies $1$-injectivity of $_{\mathbf{A}}E^{\prime
} $. \ For the sequence $\left( \ref{extended-exact-seq}\right) $ to be
exact with $n=1$, $\limfunc{coker}_{\mathbf{A}}\left( \bullet P_{1}\right) $
must be torsion-free, therefore $1$-injectivity is a weak property. \ To
summarize, injectivity of $_{\mathbf{A}}E^{\prime }$ is a stronger condition
than flatness of the dual $E_{\mathbf{A}}$. \ A convenient characterization
of dual pairs $\left( E,E^{\prime }\right) $ which are $\left( \mathcal{K}%
\right) $ over the $\mathbf{k}$-algebra $\mathbf{A}$\ (besides the case when 
$E$ is a Fr\'{e}chet space or the dual of a reflexive Fr\'{e}chet space) is
an interesting, probably difficult, and still open problem.

\begin{acknowledgement}
The author would like to thank Prof. Ulrich Oberst, whose comments were very
helpful in improving the manuscript.
\end{acknowledgement}

\end{document}